\documentclass[12pt]{amsart}
\usepackage{amsmath,amssymb,latexsym}
\usepackage{fullpage}
\usepackage{amscd}

\theoremstyle{plain}

\newtheorem{proposition}{Proposition}
\newtheorem{lemma}{Lemma}

\newtheorem{conjecture}{Conjecture}

\newtheorem{corollary}{Corollary}

\begin{document}
\date{}

\title[On the length of global integrals for $GL_n$]
{\bf On the length of global integrals for $GL_n$ }
\author{ David Ginzburg}

\begin{abstract}
In this paper we prove  Conjecture \ref{conj1}  for a set of representations of the group $GL_n({\bf A})$. This Conjecture is stated in complete generality as Conjecture 1 in \cite{G2}, and here we prove it for various cases. See
Conjecture \ref{conj2} below. First we prove it in the
case when the length of the integral is four, and then we discuss the 
general case. 

\end{abstract}

\thanks{ The author is partly supported by the Israel Science
Foundation grant number  259/14}

\address{ School of Mathematical Sciences\\
Sackler Faculty of Exact Sciences\\ Tel-Aviv University, Israel
69978 } \maketitle \baselineskip=18pt

\section{introduction}
Let $F$ denote a global field and let ${\bf A}$ denote its adele
ring. As is well known, in the Rankin-Seleberg method one writes down a global integral which depends on a complex parameter $s$, and the
basic problem is to determine when this integral is
Eulerian. One of the useful tools to study this problem is the so call dimension equation. For a definition of the dimension equation and related results and conjectures, see \cite{G1} Definition 3, \cite{G2}, \cite {G3} and \cite{G4}. Conjecture 1 as stated in \cite{G2} is one of the basic conjectures in this topic. We will now state it in the context of this paper.

For $1\le i\le l+2$, let $\pi_i$ denote $l+2$ automorphic representations of the group $GL_n({\bf A})$. Assume that $\pi_{l+1}$ is a cuspidal representation, and that $\pi_{l+2}$ is an Eisenstein series  defined on the group $GL_n({\bf A})$.

Consider the following integral,
\begin{equation}\label{global1}
\int\limits_{Z({\bf A})GL_n(F)\backslash GL_n({\bf A})}
\varphi_1(g)\varphi_2(g)\ldots\varphi_{l+1}(g)E(g,s)dg
\end{equation}
Here, $Z$ is the center of $GL_n$, and we assume that the product of all central characters of the above representations is one. Also, $\varphi_i$ is a vector in the space of $\pi_i$, and $E(g,s)$ is a certain Eisenstein series. We assume that none of the representations involved is a one dimensional representation, and we refer to the number $l+2$ as to the length of the integral. 

To define the dimension equation attached to the integral \eqref{global1}, we first define the notion of the Gelfand-Kirillov dimension of a representation. As explained in \cite{G3}, to every irreducible automorphic representation $\pi$ of $GL_n({\bf A})$ one can attach a set of unipotent orbits which we denote by ${\mathcal O}(\pi)$. As in \cite{G3} we assume that this set consists of one element. Thus,
we define the dimension of $\pi$, denoted by $\text{dim}\ \pi$, to be
a half of the number $\text{dim}\ {\mathcal O}(\pi)$. For the 
definition of the dimension of a unipotent orbit we refer to \cite{C-M}. For example, the representation $\pi_{l+1}$ is a cuspidal representation, and hence it is a generic representation. Hence ${\mathcal O}(\pi_{l+1})=(n)$ and 
$\text{dim}\ {\mathcal O}(\pi)=\frac{1}{2}n(n-1)$.
With these notations, the dimension equation is defined by, 
\begin{equation}\label{dim1}
\sum_{i=1}^{l+2}\text{dim}\ \pi_i=\text{dim}\ GL_n -1
\end{equation}
As explained in \cite{G1}, \cite{G2} and \cite{G4}, all known global unipotent integrals which are non-zero and Eulerian, do satisfy the
dimension equation \eqref{dim1}. For the definition of unipotent 
integrals see \cite{G2}. 

The main Conjecture in this topic is
\begin{conjecture}\label{conj1}
Assume that integral \eqref{global1} satisfies the dimension equation
\eqref{dim1}. Suppose that $l>1$. Then the integral is zero for all choice of data.
\end{conjecture}
In particular this Conjecture asserts that if a global unipotent integral satisfies the dimension equation, and is not zero then $l=1$. It is well known that such integrals exists. For example the Rankin product integral is such an integral. See \cite{G4} Theorem 1 for a partial classification of such integrals.

There are two main  difficulties  in studying Conjecture \ref{conj1}. The first difficulty is that it is not practical to
classify all solutions to equation \eqref{dim1}. For low values of $n$
it is not hard but the number of solutions grows quite fast. The second difficulty is the fact that for $1\le i\le l$, the representations $\pi_i$ are arbitrary and hence when unfolding the
integral and performing Fourier expansions, there are many cases to consider. 

To illustrate this , let us consider the case which motivates the integrals we study in this paper. Consider the special case of integral \eqref{global1} where the Eisenstein series is a minimal representation of $GL_n({\bf A})$. In other words, let $E(g,s)$ denote the Eisenstein series attached to the induced representation $Ind_{P({\bf A})}^{GL_n({\bf A})}\delta_P^s$. Here $P$ is the maximal parabolic of $GL_n$ whose Levi part is $GL_{n-1}\times GL_1$.  A simple unfolding process which we will perform in the next section implies that the integral
\begin{equation}\label{global2}
\int\limits_{U_n(F)\backslash U_n({\bf A})}
\varphi_1(u)\varphi_2(u)\ldots\varphi_{l}(u)\psi_U(u)du
\end{equation}
is an inner integration to integral \eqref{global1}. Here $U_n$ is the maximal unipotent subgroup of $GL_n$ and $\psi_U$ is the Whittaker character of $U_n(F)\backslash U_n({\bf A})$. For more details see Section \ref{pre}. 

More over, as we will explain below, if integral \eqref{global1}
satisfies the dimension equation \eqref{dim1}, then integral \eqref{global2} also satisfies a similar equation. Namely, we have
\begin{equation}\label{dim2}
\sum_{i=1}^{l}\text{dim}\ \pi_i=\text{dim}\ U_n=\frac{1}{2}n(n-1)
\end{equation}
Hence, in this case, Conjecture \ref{conj1} reduces to
\begin{conjecture}\label{conj2}
Suppose that integral \eqref{global2} satisfies the dimension equation
\eqref{dim2}. Then integral \eqref{global2} is zero for all choice of data.
\end{conjecture}

This Conjecture is interesting by itself. Indeed, suppose that $l=1$. Since we want to consider integrals which depends on a complex number $s$, then we take $\pi_1$ to be an Eisenstein series. Thus, integral
\eqref{global2} represents in this case the Whittaker coefficient of
an Eisenstein series. The study of these type of integrals, known as
the Langlands Shahidi integrals was studied in \cite{S}. Thus, Conjecture \ref{conj2} asserts that if $l\ge 2$, there are no nonzero integrals given by integral \eqref{global2} and satisfies \eqref{dim2}.

In studying the above Conjectures we will concentrate on the two
most important type of representations. The first type is Eisenstein series. A precise definition is given at the beginning of Section \ref{l2}. The second type of representations are what we refer to as representations of Speh type. In our context, a representation $\pi$ is a representation  of Speh type if ${\mathcal O}(\pi)=(q^m)$. Here $m$ and $q$ are two natural numbers such that $n=mq$. The
motivation is that every Speh representation, for the definition see 
\cite{J}, is such a representation. See \cite{G3} Proposition 5.3.
Notice that by our definition, every generic representation is a representation of Speh type, and it is not hard to find examples of
Eisenstein series which are also such representations. 

We are aware that there are other representations which are not of 
the types mentioned above. For example Eisenstein series at some special
values or residues of  Eisenstein series. However, the above two types
are the most important. Every representation in the discrete 
spectrum is included in them. We hope to consider the other cases in
the future.

\section{notations and preliminary results}\label{pre}

We keep the notations of the Introduction. We start by unfolding the
global integral \eqref{global1} in the case where $E(g,s)$ is the Eisenstein series defined right before integral \eqref{global2}. Assuming $\text{Re}(s)$ large this integral is equal to
\begin{equation}\label{global3}
\int\limits_{Z({\bf A})P(F)\backslash GL_n({\bf A})}
\varphi_1(g)\varphi_2(g)\ldots\varphi_{l+1}(g)f(g,s)dg
\end{equation}
Since we assume that $\pi_{l+1}$ is a cuspidal representation, we can use the well known expansion for such representations, see \cite{PS},
$$\varphi_{l+1}(g)=\sum_{\gamma\in U_n(F)\backslash P(F)}
W_{l+1}(\gamma g)$$
Here $W_{l+1}$ is the Whittaker coefficient of $\varphi_{l+1}$, defined by
$$W_{l+1}(g)=\int\limits_{U_n(F)\backslash U_n({\bf A})}
\varphi_{l+1}(ug)\psi_U(u)du$$
The character $\psi_U$ is defined as follows. Let $u=(u_{i,j})\in U_n$. Then
$\psi_U(u)=\psi(u_{1,2}+u_{2,3}+\cdots +u_{n-1,n})$. 
Plugging the above expansion in integral \eqref{global3}, we obtain
\begin{equation}\label{global4}
\int\limits_{Z({\bf A})U_n(F)\backslash GL_n({\bf A})}
\varphi_1(g)\varphi_2(g)\ldots\varphi_{l}(g)W_{l+1}(g)f(g,s)dg\notag
\end{equation}
Factoring the measure, we obtain integral \eqref{global2} as inner integration. 

Suppose that integral \eqref{global1} satisfies the dimension equation \eqref{dim1}. Since $\pi_{l+1}$ is a generic representation, then 
$\text{dim}\ \pi_{l+1}=\frac{1}{2}n(n-1)$. The Eisenstein series $E(g,s)$ used in the above integral is attached to the unipotent orbit
$(21^{n-2})$ and has dimension $n-1$. Plugging these numbers into equation \eqref{dim1} we obtain the equation \eqref{dim2}.

Let $\pi$ denote an automorphic representation, and suppose that
${\mathcal O}(\pi)=\lambda=(k_1k_2\ldots k_p)$ which is a partition of $n$. 
In other words, we have $k_i\ge k_{i+1}$ and $\sum k_i=n$. Then, as follows from \cite{C-M}, see also \cite{G4}, we have
\begin{equation}\label{dimpart}
\text{dim}\ \pi=\frac{1}{2}\text{dim}\ \lambda= \frac{1}{2}(n^2-\sum_{i=1}^p(2i-1)k_i)=\frac{1}{2}(n^2+n)-\sum_{i=1}^p ik_i
\end{equation}

In the following Lemma we compute a certain relation between the
dimensions of  certain type of partitions. We
recall that if $\lambda=(k_1k_2\ldots k_p)$ where $k_p\ge 1$, then $p$ is called the length of the partition. Also, we denote by $\lambda^t$ the transpose of the partition $\lambda$. See \cite{C-M}.

\begin{lemma}\label{lem2}
Let $\mu$ be a nontrivial  partition of $n$, and assume that $\mu^t=(m_1m_2\ldots m_r)$. Then, for any partition $\lambda$ of $n$, whose length is at most $n-m_1+1$, we have
\begin{equation}\label{len1}
\text{dim}\ \lambda+\text{dim}\ \mu> n^2-n
\end{equation}
\end{lemma}
\begin{proof}
Using \cite{C-M}, see also \cite{G3} Proposition 5.16, we have $\text{dim}\ \mu=2\sum_{1\le i<j\le r}m_im_j$. 
Using  equation \eqref{dimpart} we need to prove that 
\begin{equation}\label{len2}
I=n+\sum_{1\le i<j\le r}m_im_j-\sum_{i=1}^p ik_i>0
\end{equation}
for all partitions $\lambda=(k_1\ldots k_p)$ where $p\le n-m_1+1$. The partition $(m_11^{n-m_1})$ is a partition of length $n-m_1+1$. It is not hard to check that it satisfies inequality \eqref{len2}, and every unipotent orbit  ${\mathcal O}=(k_1'k_2'\ldots k_q')$ such that $k_1'\ge m_1$ and $q\le n-m_1+1$, is greater than or equal to  $(m_11^{n-m_1})$. If $k_1'\le m_1-1$, 
then there is a number $2\le a \le m_1$ such that ${\mathcal O}$ is greater than or equal to $(a^{p_1}(a-1)^{p_2})$
with the following conditions. First, we have 
\begin{equation}\label{len3}
p_1+p_2=n-m_1+1\ \ \ \ \ \ ap_1+(a-1)p_2=n
\end{equation}
If $a>2$, then we also have
\begin{equation}\label{len4}
\frac{am_1-(a+1)}{a-1} < n\le \frac{(a-1)m_1-a}{a-2}
\end{equation}
When $a=2$, we have the condition $n\ge 2m_1-2$.

Thus, it is enough to prove that $I>0$ for the partitions $(a^{p_1}(a-1)^{p_2})$ with the above conditions. 
To do that we compute $I$ for these partitions. It is equal to
\begin{equation}\label{len41}
n+\sum_{1\le i<j\le r}m_im_j-a\sum_{i=1}^{p_1} i -(a-1)\sum_{i=p_1+1}^{p_1+p_2}i=n+\sum_{1\le i<j\le r}m_im_j-a\sum_{i=1}^{p_1+p_2} i+\sum_{i=p_1+1}^{p_1+p_2}i\notag
\end{equation}
From this we obtain
\begin{equation}\label{len5}
I=n+\sum_{1\le i<j\le r}m_im_j+\sum_{i=p_1+1}^{p_1+p_2}i-
\frac{a}{2}(n-m_1+1)(n-m_1+2)
\end{equation}
Assume first that $a\ge 3$ and that $p_2\ge 1$. From the right hand side of the inequality \eqref{len4}, we deduce that $n-m_1+1\le \frac{m_1-2}{a-2}$. Hence, it is enough to prove that
\begin{equation}\label{len6}
I_0=n+\sum_{1\le i<j\le r}m_im_j+\sum_{i=p_1+1}^{p_1+p_2}i-
\frac{a(m_1-2)}{2(a-2)}(n-m_1+2)>0\notag
\end{equation}
Write
\begin{equation}\label{len7}
\sum_{1\le i<j\le r}m_im_j=m_1(m_2+\cdots+m_r)+\sum_{2\le i<j\le r}m_im_j=m_1(n-m_1)+\sum_{2\le i<j\le r}m_im_j\notag
\end{equation}
Plugging this into $I_0$, we deduce that $I_0$ is equal to
\begin{equation}\label{len8}
\sum_{2\le i<j\le r}m_im_j+\sum_{i=p_1+1}^{p_1+p_2}i+
\left (1-\frac{a}{2(a-2)}\right )m_1(n-m_1)+n+\frac{a(n-m_1+2)}{a-2}
-\frac{am_1}{a-2}\notag
\end{equation}
Since $a\ge 3$, then the third term from the left is positive. From the assumption that $p_2\ge 1$, and from the fact that $p_1+p_2=n-m_1+1$ we deduce that 
the second term from the left is equal to $n-m_1+1+\epsilon$ where $\epsilon\ge 0$. Hence, to conclude that $I_0>0$, it is enough to check that 
\begin{equation}\label{len9}
n+\frac{a(n-m_1+2)}{a-2}+n-m_1+1\ge \frac{am_1}{a-2}\notag
\end{equation}
This is equivalent to $n\ge (3a-2)(m_1-1)/(3a-4)$. Using the left inequality of \eqref{len4}, it is enough to prove that 
$$\frac{am_1-(a+1)}{a-1}\ge \frac{(3a-2)(m_1-1)}{(3a-4)}$$
This inequality is easy to verify. 

To conclude the case when $a\ge 3$, we still have to consider the case when $p_2=0$. When this happens, then it follows from \eqref{len3} that $a(n-m_1+1)=n$. Plugging this into \eqref{len5} we obtain 
$I=\sum m_im_j+(nm_1-n^2)/2$. Since $n=\sum m_i$, then $2\sum m_im_j-n^2=-\sum m_i^2$, and it is easy to check that $I>0$.

Finally we need to prove that when $a=2$, then $I>0$. In this case we have $n\ge 2m_1-2$, and $p_1=m_1-1$. Hence, 
\begin{equation}\label{len10}
I=\sum_{2\le i<j\le r} m_im_j +2m_1n+2m_1-\frac{1}{2}(n^2+n)
-2m_1^2-1\notag
\end{equation}
Notice that in the right hand side the first sum is over $2\le i<j\le r$. This follows from the identity $m_1(m_2+\cdots m_r)=m_1(n-m_1)$.

Let $m_i=m_1-\mu_i$ where $\mu_2\le\mu_3\le\ldots\le \mu_r$. Then $n=rm_1-\mu$ where $\mu=\mu_2+\cdots +\mu_r$. Plugging all this into the right hand side of the above equation, we obtain
$$\sum_{2\le i<j\le r}(m_1-\mu_i)(m_1-\mu_j)+2m_1(rm_1+\mu)+2m_1-
\frac{1}{2}((rm_1+\mu)^2+rm_1+\mu)-2m_1^2-1$$
Simplifying, this is equal to
$$\frac{1}{2}((r-2)m_1^2-(r-4)m_1-(\mu_2^2+\cdots +\mu_r^2)+\mu)-1$$
We have $(r-1)m_1^2-(\mu_2^2+\cdots+\mu_r^2)=\sum_{i=2}^r(m_1^2-\mu_i^2)=\sum_{i=2}^rm_i(m_1+\mu_i)$, and $\mu=(r-1)m_1-(m_2+\cdots +m_r)$. Plugging this, the above is equal to
$\frac{1}{2}(\sum_{i=2}^r(m_i(\mu_i-1)+m_1(n-2m_1+3)))-1$.
Since $n\ge 2m_1-2$, then $n-2m_1+3>0$. The first term could have some negative terms. This will happen if $m_i=m_1$ for some $i>1$. However, a direct computations shows that even in this case $I>0$.

\end{proof}

We have the following,
\begin{lemma}\label{lem1}
Suppose that $l\ge 2$. Assume that at least two of the representations $\pi_i$  are representations of Speh type. Then the dimension equation \eqref{dim2} is not satisfied.
\end{lemma}

\begin{proof}
Let $\mu=(2^{\frac{n}{2}})$ if $n$ is even, and $\mu=(2^{\frac{n-1}{2}}1)$
if $n$ is odd. Then $\mu^t=\left (\left (\frac{n}{2}\right )^2\right)$ if $n$ is even, and $\mu^t=\left (\left (\frac{n+1}{2}\right )
\left (\frac{n-1}{2}\right )\right)$ if $n$ is odd. It follows from Lemma \ref{lem2}, or by direct calculation, that equation \eqref{len1} holds with $\lambda=\mu$.

Consider  the $l$ representations $\pi_i$, and assume that $\pi_1$ and $\pi_2$ are two representations of Speh type. Then ${\mathcal O}(\pi_i)=(p_i^{q_i})$ for $i=1,2$. Hence ${\mathcal O}(\pi_i)\ge \mu$
where $\mu$ was defined above. Thus, recall that the dimension of a representation is a half of the dimension of the corresponding partition, we have
\begin{equation}\label{speh1}
\sum_{i=1}^2\text{dim}\ \pi_i=\frac{1}{2}
\sum_{i=1}^2\text{dim}\ {\mathcal O}(\pi_i)\ge \text{dim}\ \mu>\frac{1}{2}(n^2-n)\notag
\end{equation}
where the last inequality follows from equation \eqref{len1}, or by direct calculation.

\end{proof}

For a root $\gamma$ for $GL_n$ we shall denote by $\{x_\gamma(m)\}$ the one dimensional unipotent subgroup of $GL_n$.
We need the following trivial Lemma, whose proof is obtained by simple Fourier expansion,
\begin{lemma}\label{lem3}
Let $\alpha, \beta$ be two roots for the group $GL_n$ such that $\alpha+\beta$ is also a root. Let $f$ denote an automorphic function of $GL_n({\bf A})$. Consider the integral
\begin{equation}\label{int1}
\int\limits_{(F\backslash {\bf A})^2}f(x_{\alpha+\beta}(m_1)x_\beta(m_2))\psi(m_1)dm_1dm_2
\end{equation}
Then it is equal to
\begin{equation}\label{int2}
\int\limits_{\bf A}
\int\limits_{(F\backslash {\bf A})^2}f(x_\alpha(m_3)x_{\alpha+\beta}(m_1)x_\beta(m_2))\psi(m_1)dm_1dm_3dm_2\notag
\end{equation}
\end{lemma}
In particular, integral \eqref{int1} is zero for all choice of data if and only if the integral 
\begin{equation}\label{int3}
\int\limits_{(F\backslash {\bf A})^2}f(x_{\alpha+\beta}(m_1)x_\alpha(m_2))\psi(m_1)dm_1dm_2\notag
\end{equation}
is zero for all choice of data.

\section{The case $l=2$}\label{l2}
In this section we study Conjecture \ref{conj2} when $l=2$. As mentioned in the Introduction, in this paper we study this Conjecture for Speh type representations, and for Eisenstein series. Thus,  because of Lemma \ref{lem1}, we may assume that one of the two representations is an Eisenstein series. In details, for $1\le i\le r$, let $\tau_i$ denote an automorphic representation of the group $GL_{m_i}({\bf A})$ where we assume that $m_i\ge m_{i+1}$. Let $Q$ denote a parabolic subgroup of $GL_n$, whose Levi part is
$M=GL_{m_1}\times \ldots\times GL_{m_r}$. Then $\tau=\tau_1\times\tau_2\times\ldots\times\tau_r$ is a representation of $M({\bf A})$.
Denote by $E_\tau(g,\bar{s})$ the Eisenstein series of $GL_n({\bf A})$ 
attached to the induced representation $Ind_{Q({\bf A})}^{GL_n({\bf A})}\tau\delta_Q^{\bar{s}}$. We will  assume that the Eisenstein series in question is in general position. By that we mean that we are in the domain where it is given by a convergent series, and hence we can carry out an unfolding process. 
It will be convenient to separate it into two cases. 
\subsection{Eisenstein series: The Trivial Case}\label{eis1} In this subsection we assume that the representation $\tau_1$ of the group $GL_{m_1}({\bf A})$ is the trivial representation. We recall that by construction, $m_1\ge m_i$ for all $i$.

Let $\pi$ denote an irreducible automorphic representation of $GL_n({\bf A})$. The integral we consider is
\begin{equation}\label{deg1}
\int\limits_{U_n(F)\backslash U_n({\bf A})}\varphi(u)E_\tau(u,s)\psi_U(u)du
\end{equation}
We have, see \cite{G3}, $\text{dim}\ E(g,s)=\ \text{dim} U(Q)+\text{dim}\ \tau$. Hence,
$\text{dim}\ E(g,s)=\text{dim}\ \tau+\sum_{1\le i<j\le r}m_im_j$.  Then the dimension equation attached to integral \eqref{deg1} is given by
\begin{equation}\label{deg2}
\text{dim}\ \pi+\text{dim}\ \tau+\sum_{1\le i<j\le r}m_im_j=\frac{1}{2}(n^2-n)
\end{equation}

Our main result in this section is
\begin{proposition}\label{prop1}
Assume that $\pi$ satisfies equation \eqref{deg2}.  Then integral \eqref{deg1} is  zero for all choice of data. 
\end{proposition}

\begin{proof}

Unfolding the Eisenstein series, integral \eqref{deg1} is equal to
\begin{equation}\label{deg3}
\sum_{w\in Q(F)\backslash GL_n(F)/U_n(F)}\ \ 
\int\limits_{w^{-1}U_n(F)w\cap U_n(F)\backslash U_n({\bf A})}\varphi(u)f_\tau(wu,s)\psi_U(u)du
\end{equation}
The sum is finite and representatives can be taken to be Weyl elements. Factoring the measure, we obtain the integral
\begin{equation}\label{deg4}
\int\limits_{U_n^w(F)\backslash U_n^w({\bf A})}\varphi(u)\psi_U(u)du
\end{equation}
as inner integration to integral \eqref{deg3}. Here, given a Weyl element $w$, we denote $U_n^w=w^{-1}U_nw\cap U_n$. By means of Fourier expansions, we can express integral \eqref{deg4} as a sum of Fourier coefficients corresponding to a set of unipotent orbits ${\mathcal O}_1,\ {\mathcal O}_2,\ldots,{\mathcal O}_q$. Suppose that we show that
for all $1\le i\le q$ we have 
\begin{equation}\label{deg5}
\frac{1}{2}\text{dim}\ {\mathcal O}_i+\sum_{1\le i<j\le r}m_im_j>
\frac{1}{2}(n^2-n)
\end{equation}
It follows from equation \eqref{deg2} and \eqref{deg5}, that integral \eqref{deg4} is zero for all choice of data. Indeed, from 
\eqref{deg2} and \eqref{deg5} we obtain that $\frac{1}{2}\text{dim}\ {\mathcal O}_i> \text{dim}\ \pi$ for all $1\le i\le q$. Hence 
$\text{dim}\ {\mathcal O}_i > \text{dim}\ {\mathcal O}(\pi)$. By the definition of ${\mathcal O}(\pi)$, we deduce that integral \eqref{deg4} is zero for all choice of data.
But the vanishing of all these integrals   implies 
that integral \eqref{deg1} is zero for all choice of data which is what we want to prove.

We fix some notations. For $1\le i\le n$, let $V_i$ denote the unipotent subgroup of $U_n$ generated by all matrices of the form
$I_n+x_je_{i,j}$ where $i<j\le n$. Here $e_{i,j}$ is the matrix of size $n$ whose $(i,j)$ entry is one, and all other entries are zero.
We define $V_n$ to be the identity group.
Let $1\le i_1<i_2<\ldots <i_k\le n$ denote a set of natural numbers. 
Denote $V_{i_1,\ldots,i_k}=V_{i_1}V_{i_2}\ldots V_{i_k}$. Then we
claim that given a Weyl element $w\in Q\backslash GL_n/U_n$, there is a set $1\le i_1<i_2<\ldots <i_k\le n$ where $k\le n-m_1$, such that the integral 
\begin{equation}\label{deg6}
I(i_1,\ldots,i_{k-1},i_k)=\int\limits_{U_n(F)V_{i_1,\ldots,i_k}({\bf A})\backslash 
U_n({\bf A})}\varphi(u)\psi_U(u)du
\end{equation}
is an inner integration to integral \eqref{deg4}. This claim follows from the fact that every such $w$ can be chosen to be a permutation group. Hence, if $w$ has an entry one at the $(a,j_a)$ position, where 
$1\le a\le m_1$, then the group $V_{j_1,\ldots, j_{m_1}}$ is contained
in $U_n^w$. This means that there are at most $n-m_1$ indices and subgroups $V_i$
which are not contained inside $U_n^w$. It is possible that a subgroup
of these $V_i$ will be in $U_n^w$. That is why \eqref{deg6} is 
possibly an inner integration to integral \eqref{deg4}.
Thus, it is enough to prove that given a representation $\pi$ which satisfies  equation \eqref{deg2}, then for all sets $\{ i_1,\ldots, i_k\}$ as above, the integrals  $I(i_1,\ldots,i_{k-1},i_k)$ are zero for all choice of data. 

To prove that we argue by induction. First, let 
$\epsilon=(\epsilon_1,\ldots,\epsilon_{n-1})$ where $\epsilon_i=0,1$. 
Define the character $\psi_{U,\epsilon}$ of the group $U_n$ as follows. Given $u=(u_{i,j})\in U_n$, define $\psi_{U,\epsilon}(u)=
\psi(\epsilon_1 u_{1,2}+\cdots +\epsilon_{n-1} u_{n-1,n})$. We now 
define the set of Fourier coefficients
\begin{equation}\label{deg7}
I(i_1,\ldots,i_{k-1},i_k; \epsilon_{j_1},\ldots, \epsilon_{j_p})=\int\limits_{U_n(F)V_{i_1,\ldots,i_k}({\bf A})\backslash 
U_n({\bf A})}\varphi(u)\psi_{U,\epsilon}(u)du
\end{equation}
where all the $j_1,\ldots, j_p$ components of $\epsilon$ are zeros, and all other components are one. Notice that when there are no
$i_m$ indices then the integration is over $U_n(F)\backslash  U_n({\bf A})$. We shall denote these integrals by $I(\epsilon_{j_1},\ldots, \epsilon_{j_p})$. If further there are also no $\epsilon_{j_m}$ then 
integral \eqref{deg7} is the Whittaker coefficient of $\pi$. We shall
denote this integral by $I_0$.

Start with integral $I(i_1,\ldots,i_{k-1},i_k)$. By means of Fourier expansions we will prove that this integral is equal to a sum of
integrals such that the integrals $I(i_1,\ldots,i_{k-1},i_k+a)$ and
the integral $I(i_1,\ldots,i_{k-1}; \epsilon_{i_k})$
appear as inner integrations to each summand. Here $1\le a\le n-i_k$. Notice that
when $a=n-i_k$, we have $I(i_1,\ldots,i_{k-1},i_k+a)=I(i_1,\ldots,i_{k-1})$. Repeating this process with each of the integrals 
$I(i_1,\ldots,i_{k-1},i_k+a)$, we deduce that the integral 
$I(i_1,\ldots,i_{k-1},i_k)$ is  a sum of integrals such that $I(i_1,\ldots,i_{k-1})$ and  $I(i_1,\ldots,i_{k-1}; \epsilon_{i_k+a})$ appear as inner integration in each summand. Here $0\le a\le n-i_k-1$. Continuing this process with this set of integrals we finally deduce that $I(i_1,\ldots,i_{k-1},i_k)$ is a sum of integrals such   that
$I(\epsilon_{j_1},\ldots, \epsilon_{j_p})$ appear as an inner integration for some set of indices $1\le j_1<j_2<\ldots
<j_p\le n-1$ and $0\le p\le n-m_1$. Notice that the bound on
$p$ follows from the fact that the number $k$ as defined in \eqref{deg6} is bounded by $n-m_1$. 
Thus, to complete the proof we will first relate  $I(i_1,\ldots,i_{k-1},i_k)$ to the integrals $I(i_1,\ldots,i_{k-1},i_k+a)$ and $I(i_1,\ldots,i_{k-1}; \epsilon_{i_k})$ as mentioned above. Then we  prove that the integrals $I(\epsilon_{j_1},\ldots, \epsilon_{j_p})$ are zero for all choice of data.

Consider the integral 
\begin{equation}\label{deg8}
I(i_1,\ldots,i_{k-1},i_k)=\int\limits_{U_n(F)V_{i_1,\ldots,i_{k-1}}({\bf A})V_{i_k}({\bf A})\backslash 
U_n({\bf A})}\varphi(u)\psi_U(u)du
\end{equation}
where we used the fact that $V_{i_1,\ldots,i_k}=V_{i_1,\ldots,i_{k-1}}
V_{i_k}$.  Expand this integral along the one dimension unipotent subgroup $\{x_1(y_1)=I_n+y_1e_{i_k,n}\}$. Thus, $I(i_1,\ldots,i_{k-1},i_k)$ is equal to
\begin{equation}\label{deg9}
\int\int\limits_{F\backslash {\bf A}}
\varphi(ux_1(y_1))\psi_U(u)dy_1du+\sum_{\eta\in F^*}\int\int\limits_{F\backslash {\bf A}}\varphi(ux_1(y_1)t(\eta))\psi_U(u)\psi(y_1)dy_1du
\end{equation}
where the integration over $u$ is as in integral \eqref{deg8}, and $t(\eta)$ is a certain torus element. Consider each term of the right most integral in equation \eqref{deg9}. In the notations of Lemma
\ref{lem3}, we denote $x_{\alpha+\beta}(y_1)=x_1(y_1)$. Also we denote
$x_\alpha(z_1)=I_n+z_1e_{i_k,n-1}$ and $x_\beta(z_2)=I_n+z_2e_{i_k+1,n}$. Then the conditions of the Lemma hold and we can apply it. 
We repeat this process with $x_\alpha(z_1)=I_n+z_1e_{i_k,j}$ and $x_\beta(z_2)=I_n+z_2e_{i_k+n-j,n}$ in decreasing order in $j$ for all 
$i_k+1\le j\le n-2$. Then, after applying Lemma \ref{lem3} for $n-i_k-1$ times, we conjugate by the Weyl element 
$$\begin{pmatrix} I_{i_k}&&\\ &&1\\ &I_{n-i_k-1}\end{pmatrix}$$
Then it is not hard to check that we obtain the integral 
$I(i_1,\ldots,i_{k-1},i_k+1)$ as inner integration. Thus, we conclude that each summand in the right  term integral of equation \eqref{deg9} 
contains the integral $I(i_1,\ldots,i_{k-1},i_k+1)$ as inner integration. Next consider the left term integral in equation \eqref{deg9}. We expand it along the unipotent subgroup 
$\{x_2(y_2)=I_n+y_2e_{i_k,n-1}\}$. There are two terms. In the first, which corresponds to the non trivial terms in the expansion, we deduce
as in the case of $\{x_1(y_1)\}$ that the integral $I(i_1,\ldots,i_{k-1},i_k+2)$ appear as inner integration. In the second term, which corresponds to the trivial term in the expansion, we can
further expand along $\{x_3(y_3)=I_n+y_3e_{i_k,n-2}\}$. Continuing 
this process we get the first above claim, stated before integral \eqref{deg8}, regarding the induction process. Notice that the integral $I(i_1,\ldots,i_{k-1}; \epsilon_{i_k})$ is obtained by taking
in each expansion the constant term.

Finally, we need to prove that the integrals $I(\epsilon_{j_1},\ldots, \epsilon_{j_p})$ are zero for all choice of data. But this follows 
easily from Lemma \ref{lem2}. Indeed, it follows from \cite{G3}, that this Fourier coefficient  corresponds to the following unipotent orbit. Consider the numbers $\{j_1,j_2-j_1,\ldots,j_p-j_{p-1},
n-j_p\}$. Rearranging them in decreasing order, we obtain a partition $\lambda$ of $n$ whose length is $p+1\le n-m_1+1$. From Lemma \ref{lem2}, and from equation \eqref{deg5}, we deduce that 
$$\frac{1}{2}\ \text{dim}\ \lambda+ \text{dim} E_\tau(g,s)>\frac{1}{2}(n^2-n)$$ But, as explained above, this contradicts the dimension equation \eqref{deg2}.

\end{proof}

\subsection{Eisenstein Series: The Nontrivial Case}\label{eis2}
We keep the notations of the previous Subsection. The second case to consider is integral \eqref{deg1} where the Eisenstein series is the representation $E_\tau(g,\bar{s})$ as was defined in the beginning of this Section, and such that the representation $\tau_1$ is a Speh type representation. Then we may assume that the other representation in integral \eqref{deg1}, is
either a Speh type representation, or it is a similar Eisenstein series denoted by $E_\sigma(g,\bar{\nu})$.  By that we mean the following.
Let $R$ denote a parabolic subgroup of $GL_n$ whose levi part is $L=GL_{n_1}\times GL_{n_2}\times
\cdots\times GL_{n_k}$ with $n_i\ge n_{i+1}$. Let $\sigma_i$ denote an irreducible automorphic representation of $GL_{n_i}({\bf A})$, and denote $\sigma=
\sigma_1\times\ldots\times\sigma_k$. Form the Eisenstein series $E_\sigma(g,\bar{\nu})$ attached to the induced representation 
$Ind_{R({\bf A})}^{GL_n({\bf A})}\delta_R^{\bar{\nu}}$ where $\bar{\nu}$ is a multi complex variable. We also may assume that $\sigma_1$ is a Speh type representation. For if it is the trivial representation, then we may apply the argument of the Subsection \ref{eis1}. 

With these notations we prove,
\begin{proposition}\label{prop3}
With the above notations, let $\pi$ denote a Speh type representation, or  assume that $\pi=E_\sigma(g,\bar{\nu})$. Then 
$$\text{dim}\ E_\tau(g,\bar{s})+\text{dim}\ \pi>\frac{1}{2}(n^2-n)$$
In particular, the dimension equation \eqref{dim2} does not hold in this case.
\end{proposition}
\begin{proof}
We use Lemma \ref{lem2}. From \cite{G3} we deduce that the orbit 
${\mathcal O}(E_\tau(g,\bar{s}))$ is the suitable induced orbit 
as defined in \cite{C-M}. This implies that 
$${\mathcal O}(E_\tau(g,\bar{s}))={\mathcal O}(\tau_1)+\cdots+
{\mathcal O}(\tau_r)$$
The definition of addition of two partitions is given in \cite{C-M}
as follows. If $\lambda_1=(k_1k_2\ldots k_p)$ and $\lambda_2=(k_1'k_2'\ldots k_q')$, then $\lambda_1+\lambda_2=((k_1+k_1')(k_2+k_2')\ldots)$. Since $\tau_1$ is a representation of Speh type of the group $GL_{m_1}({\bf A})$, then the
length of ${\mathcal O}(\tau_1)$ is not greater then $\frac{m_1}{2}+1$. Since $\frac{m_1}{2}+1< \frac{n}{2}$ and  $m_i\le m_1$ for all $i$, then the length of ${\mathcal O}(\tau_i)$ is at most $n/2$. From this we deduce that the length of  ${\mathcal O}(E_\tau(g,\bar{s}))$ is at most $n/2$. Similarly, if $\pi=E_\sigma(g,\bar{\nu})$ or if $\pi$ is a representation of Speh type then its length is at most $n/2$. But every partition of $n$ whose length is at most $n/2$ is greater than or equal to 
$\mu=(2^{\frac{n}{2}})$ if $n$ is even, and $\mu=(2^{\frac{n-1}{2}}1)$
if $n$ is odd. Hence
$$\text{dim}\ E_\tau(g,\bar{s})+\text{dim}\ \pi=
\frac{1}{2}(\text{dim}\ {\mathcal O}(E_\tau(g,\bar{s}))+
\text{dim}\ {\mathcal O}(\pi))\ge
\text{dim}\ \mu\ge\frac{n^2-1}{2}$$
From this the proof follows.
\end{proof}

\section{The case when $l\ge 3$}\label{general}
In this section we consider the case when $l\ge 3$. Let $\pi_i$ denote $l$ automorphic representations of $GL_n({\bf A})$.
Assume that  $\pi_i=E_{\tau^{(i)}}(g,\bar{s}_i)$ for all $1\le i\le l-1$. Here, the representations $E_{\tau^{(i)}}(g,\bar{s}_i)$ were defined at the beginning of Section \ref{l2}. Assume also that 
$\tau_1^{(i)}$ is the trivial representation for all $1\le i\le l-1$.

The integral we consider is integral \eqref{global2}. We can write it as
\begin{equation}\label{gen1}
\int\limits_{U_n(F)\backslash U_n({\bf A})}
E_{\tau^{(1)}}(u,\bar{s}_1)\Phi(u)\psi_U(u)du
\end{equation}
Here $\Phi(g)=E_{\tau^{(2)}}(g,\bar{s}_2)\ldots 
E_{\tau^{(l-1)}}(g,\bar{s}_{l-1})\varphi_{\pi_l}(g)$. 

Unfold the Eisenstein series. Then carry out the same Fourier expansion process as described in the proof of Proposition \ref{prop1}. We deduce from that that integral \eqref{gen1} is zero for all choice of data if the integrals
\begin{equation}\label{gen2}
I_\Phi(\epsilon_{j_1},\ldots,\epsilon_{j_p})=
\int\limits_{U_n(F)\backslash U_n({\bf A})}
\Phi(u)\psi_{U,\epsilon}(u)du
\end{equation}
are all zero for all choice of data. All the notations were defined 
in the proof of Proposition \ref{prop1}, and we have that $p\le n-m_1^{(1)}$. It follows from the definition of $\psi_{U,\epsilon}$ that this character is not trivial at least on 
$(n-1)-(n-m_1^{(1)})=m_1^{(1)}-1$ simple roots. 

Next, in the integral $I_\Phi(\epsilon_{j_1},\ldots,\epsilon_{j_p})$
we unfold the Eisenstein series $E_{\tau^{(2)}}(g,\bar{s}_2)$ and repeat this process again. Then we deduce that integrals  \eqref{gen2}
are zero for all choice of data, if the integrals
\begin{equation}\label{gen3} 
\int\limits_{U_n(F)\backslash U_n({\bf A})}
\Phi_1(u)\psi_{U,\epsilon'}(u)du
\end{equation}
are zero for all choice of data. Here $\Phi_1(g)=E_{\tau^{(3)}}(g,\bar{s}_3)\ldots 
E_{\tau^{(l-1)}}(g,\bar{s}_{l-1})\varphi_{\pi_l}(g)$, and 
$\epsilon'=(\epsilon_1',\ldots,\epsilon_{n-1}')$ is such that at least
$(m_1^{(1)}-1)-(n-m_1^{(2)})=m_1^{(1)}+m_1^{(2)}-n-1$ of the entries are one. 

Continuing by induction we eventually get as inner integrations, the integrals \eqref{gen3}
with $\Phi_1=\varphi_{\pi_l}$ and $\epsilon'$ is a vector with at least $m_1^{(1)}+\ldots+m_1^{(l-1)}-(l-2)n-1$ entries which are equal to one. We conclude that if all such integrals are zero for all choice of data, then integral \eqref{gen1} is zero for all choice of data. 

We have
\begin{corollary}\label{cor1}
For $1\le i\le l$, let $\pi_i$ denote $l$ Eisenstein series 
$E_{\tau^{(i)}}(g,\bar{s}_i)$ as defined in the beginning of Section \ref{l2}. Assume that for all $i$ the representation $\tau_1^{(i)}$ is the trivial representation. If 
\begin{equation}\label{cond1}
\sum_{i=1}^lm_1^{(i)}\ge n(l-1)+2
\end{equation}
then integral \eqref{global2} is zero for all choice of data.
\end{corollary}
\begin{proof}
Applying the above process, the condition \eqref{cond1} implies that we obtain the integral $\int\psi(r)dr$ as inner integration to each of the integrals of the type of integral \eqref{gen3}. Here $r$ is integrated over $F\backslash {\bf A}$. From this the Corollary follows.

\end{proof}

Notice that in the above Corollary we did not assume that the dimension equation \eqref{dim2} holds.

Let $\pi_i$ denote $l$ automorphic representations of $GL_n({\bf A})$.
Because of Proposition \ref{prop3} we may assume that there is at most one representation, denoted by $\pi_l$, such that  ${\mathcal O}(\pi_l)\ge \mu=(2^{\frac{n}{2}})$ if $n$ is even, and ${\mathcal O}(\pi_1)\ge\mu=(2^{\frac{n-1}{2}}1)$ if $n$ is odd. This means that all other $l-1$ representations are of the form $E_{\tau^{(i)}}(g,\bar{s}_i)$, with the conditions  that $\tau_1^{(i)}$ is the trivial representation and that  $m_1^{(i)}> n/2$. Indeed, as argued in Proposition \ref{prop3}, if for some $i$ we have $m_1^{(i)}\le n/2$, then ${\mathcal O}(E_{\tau^{(i)}}(g,\bar{s}_i))\ge \mu$ where $\mu$ was defined above. We have,
\begin{proposition}\label{prop4}
For $1\le i\le l-1$, let $\pi_i$ denote the $l-1$ Eisenstein series defined above. 
Assume that $\pi_l=E_{\tau^{(l)}}(g,\bar{s}_l)$ and that there is a $j$ such that $\tau_j^{(l)}$ is the trivial representation. 
Assume also that the dimension equation \eqref{dim2} holds for these $l$ representations. Then
\begin{equation}\label{cond2}
\sum_{i=1}^{l-1}m_1^{(i)}+m_j^{(l)}\ge n(l-1)+2\notag
\end{equation}
In particular, integral \eqref{global2} is zero for all choice of data.
\end{proposition}
\begin{proof}
For all $1\le i\le l-1$ we have $\text{dim}\ {\mathcal O}(E_{\tau^{(i)}}(g,\bar{s}_i))=m_1^{(i)}(n-m_1^{(i)})+b_i$. we also have 
$\text{dim}\ {\mathcal O}(E_{\tau^{(l)}}(g,\bar{s}_l))=m_j^{(l)}(n-m_j^{(l)})+b_l$. Here, the $b_i$'s are some non-negative integer numbers. Thus, we can write the dimension equation \eqref{dim2} as
\begin{equation}\label{cond3}
m_j^{(l)}(n-m_j^{(l)})+\sum_{i=1}^{l-1}m_1^{(i)}(n-m_1^{(i)})+A=
\frac{1}{2}n(n-1)\notag
\end{equation}
where $A$ is a non-negative integer. With these notations, to prove the Proposition, it is enough to prove that the value at the minimum 
point of the function  $\sum_{i=1}^{l-1}m_1^{(i)}+m_j^{(l)}$, subject to the condition $m_j^{(l)}(n-m_j^{(l)})+\sum_{i=1}^{l-1}m_1^{(i)}(n-m_1^{(i)})\le \frac{1}{2}n(n-1)$, is greater or equal to 
$(l-1)n+2$. Using Lagrange multipliers, it is easy to check that the minimum is obtained when all $m_1^{(i)}$ and $m_j^{(l)}$ are equal. Denote this value by $m$. Notice, that since $m_1^{(i)}>n/2$, then
we have $m>n/2$. Thus we need to prove that $lm\ge (l-1)n+2$ if 
$lm(n-m)\le \frac{1}{2}n(n-1)$ and $m>n/2$. Solving the quadratic inequality, we obtain using that $m>n/2$ that 
$$lm\ge \frac{ln}{2} +\frac{1}{2}\left [ (l^2-2l)n^2+2ln\right ]^{1/2}$$
It is easy to check that the right hand side is greater or equal to 
$(l-1)n+2$.

\end{proof}

Next we consider the case when $\pi_l$ is a representation of Speh type. Thus we assume that $n=pq$ and that ${\mathcal O}(\pi_l)=(p^q)$.
Since $(p^q)\ge \mu$, then arguing as after Corollary \ref{cor1}, we may assume that for all $1\le i\le l-1$ the representations $\pi_i$ are the Eisenstein series $E_{\tau^{(i)}}(g,\bar{s}_i)$ such that $\tau_1^{(1)}$ is the trivial representation of $GL_{m_1^{(i)}}({\bf A})$, and $m_1^{(i)}>n/2$. From equation \eqref{dimpart} we have  $\text{dim}\ \pi_l=\frac{1}{2}n(n-q)$. Hence the dimension equation in this case is
\begin{equation}\label{cond4}
\frac{1}{2}n(q-1)+\sum_{i=1}^{l-1}m_1^{(i)}(n-m_1^{(i)})+A=
\frac{1}{2}n(n-1)
\end{equation}
where $A$ is a non-negative integer. In a simialr way as in the previous case, this time we want to minimize $\sum_{i=1}^{l-1}m_1^{(i)}-(l-2)n-1$ subject to the conditions 
$\sum_{i=1}^{l-1}m_1^{(i)}(n-m_1^{(i)})\le \frac{1}{2}n(q-1)$ and $m_1^{(i)}>n/2$. The first term is the expression which appears in equation \eqref{cond1} with $l-1$ instead of $l$. The above inequality is derived from the equation \eqref{cond4}. As in Proposition \ref{prop4}, the minimum of the function $\sum_{i=1}^{l-1}m_1^{(i)}-(l-2)n-1$ is derived when all $m_1^{(i)}$ are equal, and so we need 
to minimize the function $(l-1)m-(l-2)n-1$ with the condition 
$(l-1)m(n-m)\le \frac{1}{2}n(q-1)$. The solution in $m$ of the 
quadratic inequality which satisfies $m>n/2$ is
$$m=\frac{n}{2}+\frac{1}{2(l-1)}\left [(l-1)^2n^2-2n(l-1)(q-1)
\right ]^{1/2}$$
Plugging this value in $(l-1)m-(l-2)n-1$ it is easy to prove that it is greater than $n-q+1$. 

From all this we deduce that after unfolding the Eisenstein series, as we did in the case of integral \eqref{gen1}, we obtain as inner integration to integral \eqref{gen1}, the integrals
\begin{equation}\label{cond5} 
\int\limits_{U_n(F)\backslash U_n({\bf A})}
\varphi_{\pi_l}(u)\psi_{U,\epsilon'}(u)du
\end{equation}
where the vector $\epsilon'$ has at least $n-q+1$ nonzero entries. Hence, it follows from Proposition 5.3 in \cite{G3}, and from the assumption that ${\mathcal O}(\pi_l)=(p^q)$, that integral \eqref{cond5} is zero for $\psi_{U,\epsilon'}$ as above. Indeed, the
Fourier coefficient corresponding to the unipotent orbit $(p^q)$ is
given by integral \eqref{cond5} where $\epsilon'$ has exactly $n-q$
nonzero entries. It is not hard to check that if $\epsilon'$ has at
least $n-q+1$ nonzero entries, we obtain a Fourier coefficient corresponding to a unipotent orbit which is greater than or not related to $(p^q)$. We summarize,
\begin{proposition}\label{prop5}
For $1\le i\le l-1$, let $\pi_i$ denote the $l-1$ Eisenstein series defined above. 
Assume that $\pi_l$ is a representation of Speh type. 
Assume also that the dimension equation \eqref{dim2} holds for these $l$ representations. Then integral \eqref{global2} is zero for all choice of data.
\end{proposition}

\end{document}